\documentclass[11pt]{amsart}

\usepackage{algorithm}
\usepackage[noend]{algpseudocode}
\usepackage{graphicx}
\usepackage{amsfonts,amsmath,amssymb}
\usepackage[margin=1in]{geometry}
\newtheorem{theorem}{Theorem}[section]
\newtheorem{proposition}[theorem]{Proposition}
\newtheorem{conjecture}[theorem]{Conjecture}
\newtheorem{observation}[theorem]{Observation}

\newtheorem{corollary}[theorem]{Corollary}
\newtheorem{definition}[theorem]{Definition}

\newcommand{\ql}{\ensuremath{q_{\text{list}}}}

\newcommand{\bE}{\ensuremath{\mathbf{E}}}

\begin{document}

\title[Edge-coloring linear hypergraphs with medium-sized edges]{Edge-coloring linear hypergraphs with medium-sized edges}
\author[Vance Faber and David G. Harris]{
{\sc Vance Faber}$^{1}$
\and
{\sc David G.~Harris}$^{2}$
}

\setcounter{footnote}{0}

\addtocounter{footnote}{1}
\footnotetext{IDA/Center for Computing Sciences, Bowie MD 20707. Email: \texttt{vance.faber@gmail.com}}

\addtocounter{footnote}{1}
\footnotetext{Department of Computer Science, University of Maryland, 
College Park, MD 20742. 
Email: \texttt{davidgharris29@gmail.com}.}

\date{}
\maketitle

\begin{abstract}
Motivated by the Erd\H{os}-Faber-Lov\'{a}sz (EFL) conjecture for hypergraphs, we consider the list edge coloring of linear hypergraphs. We show that if the hyper-edge sizes are bounded between $i$ and $C_{i,\epsilon} \sqrt{n}$ inclusive, then there is a list edge coloring using $(1 + \epsilon) \frac{n}{i - 1}$ colors. The dependence on $n$ in the upper bound is optimal (up to the value of $C_{i,\epsilon}$).
\end{abstract}

\section{Introduction}
Let $H = (V,E)$ be a hypergraph with $n = |V|$ vertices. Each edge $e \in E$ can be regarded as a subset of $V$. We say that $H$ is \emph{linear} if $|e \cap e'| \leq 1$. We define the \emph{minimum rank} of $H$ as $\rho = \min_{e \in E} |e|$ and similarly the \emph{maximum rank} of $H$ as $P = \max_{e \in E} |e|$. A graph is a special case with $\rho = P = 2$.

For any hypergraph $H$, one may define the \emph{line graph} of $H$, to be an (ordinary) graph $L(H)$ on vertex set $E$, with an edge $\{ e_1, e_2 \}$ in $L(H)$ iff $e_1 \cap e_2 \neq \emptyset$.  The edge chromatic number $q(H)$ (respectively list edge chromatic number $\ql(H)$) is the chromatic number $\chi$ (respectively list chromatic number $\chi_{\text{list}}$) of $L(H)$.

A long-standing conjecture, known now as the Erd\H{o}s-Faber-Lov\'{a}sz conjecture, can be stated as
\begin{conjecture}[EFL] 
Let $H$ be a linear hypergraph with $n$ vertices and no rank-1 edges. Then $q(H) \leq n$. 
\end{conjecture}

There has been partial progress to proving this result; in \cite{kahn2} Kahn showed that $q(H) \leq n+o(n)$. See \cite{arroyo} for a more recent review of results.

In this paper, we will show a related result for hypergraphs in which the edges all have medium size. More specifically, we show the following:
\begin{theorem}
\label{main-thm}
For any integer $i \geq 3$ and any $\epsilon > 0$, there exists some value $C_{i, \epsilon} > 0$ with the following property. For any linear hypergraph $H$ on $n$ vertices, with minimum rank $\rho \geq i$ and maximum rank $P \leq C_{i,\epsilon} \sqrt{n}$, it holds that
$$
\ql(H) \leq (1 + \epsilon) \frac{n}{i-1}
$$
\end{theorem}

In particular, the EFL conjecture holds for hypergraphs of minimum rank $\rho \geq 3$ and maximum rank $P \leq C \sqrt{n}$, for some universal constant $C = C_{3,1}$. By way of comparison, \cite{arroyo2} showed that the EFL conjecture holds for hypergraphs of minimum rank $\rho \geq \sqrt{n}$. However, Theorem~\ref{main-thm} can be significantly stronger than EFL in cases where $\rho$ is large. We also note that the dependence of $P$ on $n$ is optimal, up to the value of the constant $C_{i, \epsilon}$; see Section~\ref{lb-sec} for further details.

\section{Preliminaries}
\subsection{Notation}
To simplify some notation, we define the truncated logarithm by
$$
\log(x) = \begin{cases}
\ln(x) & \text{if $x \geq e$} \\
1 & \text{otherwise}
\end{cases}
$$
Thus, $\log(x) \geq 1$ for all $x \in \mathbf R$. 

If $H = (V,E)$ is a hypergraph and $E' \subseteq E$, then $H(E')$ is the hypergraph $(V, E')$. If $V' \subseteq V$, then $H[V']$ is the induced hypergraph $(V', E')$, where $E'$ is the set of edges involving only vertices of $V'$. If the hypergraph $H$ is understood, then we sometimes write $\ql(E')$ as shorthand for $\ql(H(E'))$.
\subsection{Background facts}
Our proof is based upon two powerful theorems for list-coloring hypergraphs. The first is due to  Kahn \cite{kahn}, restated in terms of linear hypergraphs.
\begin{theorem}[\cite{kahn}]
\label{kahn-thm}
For every $\epsilon > 0$ and integer $P > 1$ there exists a constant $c = c(\epsilon, P)$, such that any linear hypergraph $H$ with maximum rank $P$ and maximum degree $\Delta \geq c$ satisfies
$$
\ql(H) \leq \Delta (1 + \epsilon)
$$
\end{theorem}

Observe that if $P \leq O(1)$, then Theorem~\ref{kahn-thm} immediately shows Theorem~\ref{main-thm}, namely that $\ql(H) \leq \frac{n}{\rho-1} (1 + \epsilon)$. Indeed, Kahn's asymptotic proof of EFL \cite{kahn2} uses this proof strategy. However, Theorem~\ref{kahn-thm} does not give useful information when $P$ is increasing with $n$. 

In this second case, we follow a strategy of \cite{faber1} and use another graph property, \emph{local sparsity}, based on the triangle counts in the line graph.
\begin{definition}
A \emph{triangle} in a graph $G$ is a set of three vertices $v_1, v_2, v_3$ such that $(v_1, v_2) \in E, (v_1, v_3) \in E, (v_2, v_3) \in E$.
\end{definition}
\begin{theorem}[\cite{vu}]
  \label{vu-thm}
  Suppose that every vertex of $G$ has degree at most $d$, and every vertex participates in at most $f$ triangles. Then
$$
\chi_{\text{list}}(G) \leq O( \frac{d}{\log(d^2/f)} )
$$
\end{theorem}

We note that \cite{aks} had earlier showed a similar bound $\chi(G) \leq O( \frac{d}{\log(d^2/f)} )$ for the ordinary chromatic number. However, our proofs really need list-coloring as a subroutine: even if our goal was only to bound the ordinary chromatic index $q(H)$, we would still need the list-coloring provided by Theorem~\ref{vu-thm}.

\section{Main result}
In this section, we prove Theorem~\ref{main-thm}. We let $L$ denote the line graph of $H$. Also, we let $\epsilon > 0$ be an arbitrary fixed quantity which we view as a \emph{constant} throughout the proof.

\begin{observation}
\label{obs1}
Every vertex $v \in H$ is in at most $\frac{n-1}{\rho-1}$ edges of $H$.
\end{observation}
\begin{proof}

Since $H$ is linear, the edges containing $v$ do not share any other vertices.
\end{proof}

\begin{proposition}
\label{gap-prop}
Suppose we are given a linear hypergraph $H = (V,E)$ and a partition $E = E_1 \sqcup E_2$. 
Let $P_1 = \max_{e_1 \in E_1} |e_1|$ and let $\rho_2 = \min_{e_2 \in E_2} |e_2|$.

Then
$$
\ql(H) \leq \max \Bigl(  \ql(E_2), \ql(E_1) + \frac{ (n-1) P_1}{\rho_2 - 1}  \Bigr)
$$
\end{proposition}
\begin{proof}
Suppose each edge has a palette of size $Q \geq \ql(E_2)$. Select an arbitrary list-coloring of $H(E_2)$. Consider the residual palette for each edge $e \in E_1$; that is, the palette available to $e$ after removing the colors selected by edges $f \in E_2$ with $f \cap e \neq \emptyset$. By Observation~\ref{obs1}, each vertex $v$ is in at most $\frac{n-1}{\rho_2-1}$ edges in $E_2$, and thus each edge $e \in E_1$ touches at most $\frac{(n-1) P_1}{\rho_2 - 1}$ edges in $E_2$. Thus, each edge in $E_1$ has a residual palette of size at least $Q - \frac{(n-1) P_1}{\rho_2 - 1}$.

As long as $Q - \frac{(n-1) P_1}{\rho_2 - 1} \geq \ql(E_1)$, we can list-color $H(E_1)$ with the residual palette, thus giving a full list-coloring of $H$. Thus, our coloring procedure succeeds as long as
$$
Q \geq  \ql(E_2) \qquad \text{and} \qquad Q \geq \frac{(n-1) P_1}{\rho_2 - 1} + \ql(E_1)
$$
\end{proof}

Given a hypergraph $H$, we let $A_i$ denote the set of edges $e \in E$ such that $2^i \leq |e| < 2^{i+1}$. Note that each edge $e \in E$ corresponds to a vertex of $L$, and for any $U \subseteq E$ we denote by $L[U]$ the induced subgraph of $L$ on the edges of $U$. 

\begin{proposition}
\label{prop1}
For any integer $i \geq 1$, we have
$$
\ql(A_i) \leq O\bigl( \frac{n}{i} + \frac{n}{\log(n/2^{2 i}))} \bigr)
$$
\end{proposition}
\begin{proof}
By Observation~\ref{obs1}, every vertex in $H$ touches $O(n/2^i)$ edges in $A_i$. As each edge in $A_i$ contains at most $2^{i+1}$ vertices, it follows that $L[A_i]$ has maximum degree $O(n)$.

Now, consider some edge $e \in A_i$; we want to count how many triangles of $L$ it participates in. There are two types of triangles. The first involves three edges around a single vertex; for each $v \in e$ there are at most $n/2^i$ choices for the other two edges, giving a total triangle count of at most $2^{i+1} \times (n/2^i)^2 = O(n^2/2^i)$.

The second type of triangle involves three edges, each intersecting at a distinct point. There are $O(n)$ choices for the second edge $e'$. By linearity, for any given choice of vertices $v \in e, v' \in e'$ there is at most one possible choice for the third edge $e''$ intersecting $e$ at $v$ and $e'$ at $v'$. Hence, the total number of triangles of this second type is at most $n \times (2^{i+1})^2 \leq O(n 2^{2i})$.

Now, applying Theorem~\ref{vu-thm}, we see that
\begin{align*}
\ql(A_i) &\leq O( \frac{n}{\log( \frac{n^2}{O(n^2/2^i + n 2^{2i})})} ) \leq O( \frac{n}{i} + \frac{n}{\log(n/2^{2i})})
\end{align*}
\end{proof}

The remainder of this proof will be expressed in terms of an integer parameter $k$ (which does not depend on $n$), and which we will set later in the construction.

\begin{observation}
\label{obs2}
For each integer $k \geq 1$, there is some integer $N_k$, such that whenever $n \geq N_k$ it holds that
$$
\ql(A_1 \cup A_2 \cup \dots \cup A_k) \leq (1 + \epsilon) \frac{n}{\rho-1}
$$
\end{observation}
\begin{proof}
By Observation~\ref{obs1}, $L[A_1 \cup A_2 \cup \dots \cup A_k]$ has maximum degree $\frac{n-1}{\rho-1}$ and maximum rank $2^{k+1} - 1$. Now apply Theorem~\ref{kahn-thm}.
\end{proof}

\begin{corollary}
\label{cor1}
If $i \geq k$ and $P \leq \sqrt{n e^{-k}}$, then $\ql(A_i) \leq O(n/k)$.
\end{corollary}
\begin{proof}
This is vacuously true for $i \geq \log_2 P$, since then $A_i$ is empty. Otherwise, Proposition~\ref{prop1} gives $
\ql(A_i) \leq O( \frac{n}{i} + \frac{n}{\log(n 2^{-2i})} ) \leq O( \frac{n}{k} + \frac{n}{\log(n P^{-2})}) \leq O(n/k)$.
\end{proof}
\begin{proposition}
\label{prop2}
Suppose $i \geq k$ and $P \leq \sqrt{n e^{-k}}$. Let $x = \lceil \log_2 k \rceil$. Then 
$$
\ql(A_i \cup A_{i+x} \cup A_{i+2x} \cup A_{i+3x} \dots) \leq O(n/k)
$$
\end{proposition}
\begin{proof}
Let $s = \lceil \log_2 n \rceil$. For each integer $j \geq 0$, define
$$
B_j = A_{i + s x} \cup A_{i + (s-1) x} \cup A_{i + (s-2) x} \cup \dots \cup A_{i + (j+1) x} \cup A_{i + j x}
$$

We will prove that $\ql(B_j) \leq c n/k$ for some sufficiently large constant $c$ and $j \geq 0$, by induction downward on $j$. 

When $j = s$, this is vacuously true as $B_j = \emptyset$. For the induction step, we use Proposition~\ref{gap-prop} using the decomposition $E_1 = A_{i+ j x}$ and $E_2 = A_{i+(j+1) x} \cup A_{i+(j+2) x} \cup \dots \cup A_{i+ s x} = B_{j+1}$.

Note that $P_1 \leq 2^{i+j x+1} - 1$ and $\rho_2 \geq 2^{i + (j+1) x} $. By induction hypothesis, $\ql(E_2) = \ql(B_{j+1}) \leq c n / k$. By Corollary~\ref{cor1}, $\ql(E_1) = \ql(A_{i + j x}) \leq c' n/k$, for some constant $c' \geq 0$. This gives
\begin{align*}
\ql(B_j) &\leq \max( \ql(E_2), \ql(E_1) + \frac{ (n-1) P_1}{\rho_2 - 1} ) \\
&\leq \max( c n/k, c' n/k + \frac{(n-1) (2^{i+j x + 1} - 1) }{2^{i+(j+1) x} - 1} ) \\
&\leq \max( c n/k, c' n/k + 4 n/k )
\end{align*}

which is at most $c n / k$ when $c \geq c' + 4$.
\end{proof}

\begin{proposition}
\label{prop3}
For each integer $k \geq 1$, there is some integer $N'_k$  such that  whenever $n > N'_k$ and $P \leq \sqrt{n e^{-k}}$ it holds that
$$
\ql(A_k \cup A_{k+1} \cup A_{k+2} \cup \dots )\leq O( \frac{n \log k}{k} )
$$
\end{proposition}
\begin{proof}
Suppose every edge begins with a palette of size $\frac{c n \log k}{k}$ for some constant $c$. Let $x = \lceil \log_2 k \rceil$. We randomly partition the colors into $x$ classes. For any class $i$ and edge $e$, let $Q_{i,e}$ denote the number of class-$i$ colors in the palette of edge $e$.

We have $\bE[Q_{i,e}] \geq \frac{c n \log k}{k x} \geq \frac{c n}{2 k}$. Furthermore, $Q_{i,e}$ is the sum of independent random variables (whether each color in the palette of $e$ goes into $Q_{i,e}$), and so by Chernoff's bound we have $$
\Pr( Q_{i,e} \leq \frac{c n}{4 k}) \leq e^{-\Theta(n/k)}
$$
For fixed $k$, this is smaller than $n^2 k/2$ for sufficiently large $n$. Since there are at most $n^2$ edges and $k$ classes, by the union bound there is a positive probability that every edge has at least $\frac{c n}{4 k}$ colors of each class in its palette.

Next, for each $i$ in the range $i = 1, \dots, x$, we use the class-$i$ colors to color $L[A_{k+i} \cup A_{k+i+x} \cup A_{k+i+2 x} \cup \dots ]$. By Proposition~\ref{prop2} this succeeds for $c$ sufficiently large.
\end{proof}

We are now ready to prove Theorem~\ref{main-thm}.
\begin{proof}[Proof of Theorem~\ref{main-thm}]
We will prove that $\ql(H) \leq \frac{n}{i-1} (1 + O(\epsilon))$; the result follows easily by rescaling.

By Proposition~\ref{prop3}, there is some constant $c$ such that for any integer $k > 1$ and $n > N'_k$ we have $\ql(A_k \cup A_{k+1} \cup A_{k+2} \cup \dots )\leq \frac{c n \log k}{k}$. We select $k$ sufficiently large so that $$
\frac{c \log k}{k} \leq \frac{\epsilon}{i-1}
$$

By Observation~\ref{obs2}, for $n > N_k$ we have $$
\ql(A_1 \cup A_2 \cup \dots \cup A_{k-1}) \leq \frac{n (1+\epsilon)}{i-1}
$$

Now, suppose every edge has a palette of size $Q = \frac{(1 + 3 \epsilon) n}{i-1}$. Randomly partition the colors into two classes, where a color goes into class I with probability $p = \frac{1.5 n \epsilon}{Q}$ and goes into class II with probability $1-p$. A Chernoff bound argument similar to Proposition~\ref{prop3} shows that for $n > N''$, with positive probability every edge has at least $n \epsilon$  class-I colors and $\frac{n}{i-1}(1 + \epsilon)$ class-II colors. In such a case, we can color $L[A_k \cup A_{k+1} \cup A_{k+2} \cup \dots ]$ using class-I colors and $L[A_1 \cup A_2 \cup \dots \cup A_{k-1}]$ using class-II colors.

So far, we have shown that there is an integer $k$ and integer $M = \max(N_k, N'_k, N''_k)$ such that whenever $P \leq \sqrt{n e^{-k}}$ and $n > M$ that $\ql(H) \leq \frac{n}{i-1}(1 + O(\epsilon))$. Set $C_{i, \epsilon} = \min(e^{-k/2}, \frac{1}{\sqrt{M}})$. 

When $P \leq C_{i, \epsilon} \sqrt{n}$, then $P \leq \sqrt{n e^{-k}}$. Also, when $P \leq C_{i, \epsilon} \sqrt{n}$, we have $P \leq \sqrt{\frac{n}{M}}$. Since $P \geq i \geq 2$, this implies $n \geq 4 M$. So $\ql(H) \leq \frac{n}{i-1}(1 + O(\epsilon))$.
\end{proof}

\section{Dependence of $n$ on $P$}
\label{lb-sec}

In this section, we show (roughly speaking) that Theorem~\ref{main-thm} must include a condition of the form $P \leq C_{i,\epsilon} \sqrt{n}$ where $C_{i,\epsilon} \leq  \sqrt{\frac{1+\epsilon}{i-1}}$.
\begin{proposition}
\label{lb-prop1}
For every real number $x$ in the range $(0,1)$ and every $\delta > 0$, there is some integer $N$ such that, for all $n \geq N$, there are linear hypergraphs $H$ on $n$ vertices which satisfy the following properties:
\begin{enumerate}
\item $\ql(H) > x n$
\item $P = \rho = r$
\item $\sqrt{n x} \leq r \leq (1+\delta) \sqrt{n x}$
\end{enumerate} 
\end{proposition}
\begin{proof}
Consider a finite projective plane with parameter $q$; this gives a linear hypergraph $H'$ with $q^2 + q + 1$ vertices and every edge of rank $r = q + 1$, and $\ql(H') = q^2 + q + 1$. Assuming that $q^2 + q + 1 \leq n$, we can form $H$ from $H'$ by adding $n - (q^2 + q + 1)$ isolated vertices.

Let $u = \sqrt{n x}$ and let $q$ be the smallest prime with $q \geq u$. For $n$ (and hence $u$ sufficiently large), this $q$ satisfies $q \leq u + u^{\theta}$ for some constant $\theta < 1$. 
In particular, $q^2 + q + 1 \leq n x + o(n) \leq n$ for $n$ sufficiently large. Our choice of $q$ ensures that that $\ql(H') = q^2 + q + 1 \geq u^2 + u + 1 > x n$. Also, $r > \sqrt{ n x}$ and $r/\sqrt{n} \rightarrow \sqrt{x}$ as $n \rightarrow \infty$.
\end{proof}

\begin{proposition}
\label{lb-prop2}
For any integer $i \geq 3$ and $\epsilon \in (0,1)$, the term ``$C_{i,\epsilon} \sqrt{n}$'' in Theorem~\ref{main-thm} cannot be replaced by any expression of the form $f(i, \epsilon, n)$, where for any fixed $i, \epsilon, \delta$ there are infinitely many $n$ with
$$
f(i,\epsilon,n) > (1 + \delta) \sqrt{\frac{(1+\epsilon) n}{i-1}}
$$
\end{proposition}
\begin{proof}
Apply Proposition~\ref{lb-prop1} with $x = (1+\epsilon)/(i-1)$; note that since $i \geq 3$ and $\epsilon < 1$ we have $x \in (0,1)$. This ensures that for all $n > N$, there is a hypergraph $H_n$ with $\ql(H_n) > x n$ and rank $\rho = P = r$ for $\sqrt{n x} \leq r \leq (1+\delta) \sqrt{n x}$.

Thus, for $n$ sufficiently large, we have $\rho \geq i$. Also, for sufficiently large $n$, we have $P \leq (1+\delta) \sqrt{n x} = (1+\delta) \sqrt{ \frac{(1+\epsilon) n}{i-1}}$. Thus, if $f(i,\epsilon,n) > (1 + \delta) \sqrt{\frac{(1+\epsilon) n}{i-1}}$ for infinitely many $n$, then for infinitely many $n$ we would have $\ql(H_n) \leq \frac{(1+\epsilon) n}{i-1} = x n$, a contradiction.
\end{proof}

\section{Acknowledgments}
Thanks to the anonymous reviewers for helpful suggestions and corrections.

\end{document}